\documentclass[reqno,english]{amsart}%
\usepackage{amsfonts,amsmath,latexsym,verbatim,amscd,mathrsfs,color,array}
\usepackage[colorlinks=true]{hyperref}
\usepackage{amsmath,amssymb,amsthm,amsfonts,graphicx,color}
\usepackage{amssymb}
\usepackage{pdfsync}
\usepackage{epstopdf}
\usepackage{cite}
\usepackage{graphicx}
\usepackage{amsmath}
\usepackage{amsfonts}%
\providecommand{\U}[1]{\protect\rule{.1in}{.1in}}

\numberwithin{equation}{section}

\newtheorem{thm}{Theorem}[section]
\newtheorem{thmx}{Theorem}

\newtheorem{lem}[thm]{Lemma}
\newtheorem{prop}[thm]{Proposition}

\newtheorem{rem}{Remark}

\newcommand{\D}{\Delta}

\newcommand{\R}{\mathbb{R}}
\newcommand{\ve}{\varepsilon}

\newcommand{\A}{\mathcal{A}}

\begin{document}

 \title[Toda system with multiple sources]{On $SU(3)$ Toda system with multiple singular sources}

\thanks{The first author is supported by the SNSF   Grant No. P2BSP2-172064}
\thanks{ The  third author  is partially supported by NSERC}


\author[A. Hyder]{Ali Hyder}
\address{\noindent Department of Mathematics, University of British Columbia,
Vancouver, B.C., Canada, V6T 1Z2}
\email{ali.hyder@math.ubc.ca}
\author[CS Lin]{Changshou Lin}
\address{\noindent  Taida Institute of Mathematics, Center for Advanced study in Theoretical Science, National Taiwan University, Taipei, Taiwan.}
\email{cslin@math.ntu.edu.tw}
\author[J. Wei]{Juncheng Wei}
\address{\noindent Department of Mathematics, University of British Columbia,
Vancouver, B.C., Canada, V6T 1Z2}
\email{jcwei@math.ubc.ca}

\begin{abstract}
  We consider the singular $SU(3)$ Toda system with multiple singular sources      \begin{align*} \left\{\begin{array}{ll}-\Delta w_1=2e^{2w_1}-e^{w_2}+2\pi\sum_{\ell=1}^m\beta_{1,\ell}\delta_{P_{\ell}}\quad\text{in }\mathbb{R}^2\\ \rule{0cm}{.5cm} -\Delta w_2=2e^{2w_2}-e^{w_1}+2\pi\sum_{\ell=1}^m\beta_{2,\ell}\delta_{P_{\ell}}\quad\text{in }\mathbb{R}^2 \\ w_i(x)=-2\log|x|+O(1)\quad\text{as }|x|\to\infty,\, i=1,2, \end{array}\right.\end{align*}  with $m\geq 3$ and $\beta_{i,\ell}\in [0,1)$. We prove the existence and non-existence results under suitable assumptions on $\beta_{i,\ell}$.  This generalizes Luo-Tian's \cite{Luo-Tian} result for a singular Liouville equation in $\mathbb{R}^2$. We also study existence results for a higher order singular Liouville equation in $\mathbb{R}^n$.
  \end{abstract}

\maketitle

\section{Introduction}
We consider the following  singular $SU(3)$ Toda system with multiple singular sources
 \begin{align}\label{eq-w-toda} \left\{\begin{array}{ll}-\D w_1=2e^{2w_1}-e^{w_2}+2\pi\sum_{\ell=1}^m\beta_{1,\ell}\delta_{P_{\ell}}\quad\text{in }\R^2\\ \rule{0cm}{.5cm} -\D w_2=2e^{2w_2}-e^{w_1}+2\pi\sum_{\ell=1}^m\beta_{2,\ell}\delta_{P_{\ell}}\quad\text{in }\R^2, \end{array}\right.\end{align} where $P_1,\dots,P_m$ are distinct points in $\R^2$,  $\beta_{i,\ell}\in [0,1)$ and $\delta_P$ denotes the Dirac measure  at $P$ (notice that  source terms are written with a plus sign).  When $w_1=w_2, \beta_{1,l}=\beta_{2,l}=\beta_l$, the above system reduces to the singular Liouville equation \begin{align}\label{eq-w} -\D w=e^{2w}+2\pi\sum_{\ell=1}^m\beta_\ell\delta_{P_\ell}\quad\text{in }\R^2.\end{align} The Toda system \eqref{eq-w-toda} and the Liouville equation \eqref{eq-w} have been widely studied in the literature due to  its important role in geometry and  mathematical physics. For instance, Eq. \eqref{eq-w} is related to the problem of  prescribing Gaussian curvature on surfaces with conical singularity,  and  abelian gauge in Chern-Simons theory \cite{Bar-Lin, Bar-Tar, Ta, Tr}. The Toda system \eqref{eq-w-toda} appears in the description of holomorphic curves in $\mathbb{CP}^3$ \cite{Bolton,Calabi,CW,Doliwa}, and in the non-abelian Chern-Simmon theory \cite{Dunne,NT,Yang}.  For classification and blow-up analysis to the (singular)  Liouville equation and the $SU(n)$ Toda system we refer the reader to \cite{BJMR, Bar-Tar, Bar-Tar2,BM,Brezis-Merle,CL, Carlotto,H-class, HMM,JLW,JW,LWYZ,Lin,Lucia-Nolasco,MarClass,PT,LWY,LWZ,LNW} and the references therein.

Luo-Tian \cite{Luo-Tian} gave a necessary and sufficient condition for the existence of singular metric with three or more conical singularities on the $2$-sphere, whose equivalent statement on $\R^2$ is the following theorem:
\begin{thmx}[\cite{Luo-Tian}] \label{Luo-Tian} Let $m\geq 3$.   Let $P_1,\dots,P_m$ be $m$ distinct points in $\R^2$. Then there exist  continuous functions $h_\ell$ around $P_\ell$ for $\ell=1,\dots,m$, a bounded continuous function $h_{m+1}$ outside a compact set, and a solution $w$ to  \begin{align}\label{eq-w-3}\left\{ \begin{array}{ll} -\D w=e^{2w}&\quad\text{in }\R^2\setminus\{P_1,P_2,\dots,P_m\} \\ w(x)=-\beta_\ell\log|x-P_\ell|+h_\ell(x)&\quad\text{around each }P_\ell \\ \rule{0cm}{.5cm}w(x)=-2\log|x|+h_{m+1}(x)&\quad\text{as }|x|\to\infty \\\beta_\ell\in (0,1) &\quad \ell=1,\dots,m \end{array}\right. \end{align} if and only if \begin{align}\label{cond-beta} \sum_{\ell=1}^m\beta_i<2\quad\text{and }\sum_{\ell\neq j}\beta_\ell>\beta_j\quad\text{for every }j=1,2,\dots,m.\end{align}  Moreover, the solution is unique. \end{thmx}

Troyanov \cite{Tr2} studied  singular metrics with  $2$ singulirities (i.e., $m=2$) and  constant curvature $1$ on the $2$-sphere, and showed that  the order of both singularities are equal (i.e., $\beta_1=\beta_2<1$). A necessary and sufficient condition on $\{\beta_1,\beta_2,\beta_3\}\subset(-\infty,1)$ for the existence of singular metric on the  $2$-sphere  has been given in  \cite{Eremenko,Umehara}. See also \cite{BMM,BJMR,Malchiodi-Ruiz} and the references therein for various existence results on  compact surfaces.

In this paper we  study Problem \eqref{eq-w-3}  in the context of $SU(3)$ Toda system. More precisely, we prove existence and non-existence of solutions $(w_1,w_2)$ to \eqref{eq-w-toda}  satisfying  \begin{align}\label{behavior-w-toda}\left\{ \begin{array}{ll}  w_i(x)=-\beta_{i,\ell}\log|x-P_\ell|+h_{i,\ell}&\quad\text{around each point }P_\ell\\ \rule{0cm}{.5cm}w_i(x)=-2\log|x|+h_{i,m+1}&\quad\text{as }|x|\to\infty\\ h_{i,\ell}\text{ is  continuous in a neighborhood of }P_\ell,  \end{array}\right. \end{align}  for $i=1,2$ and $\ell=1,\dots,m$, and  $h_{i,m+1}$ is bounded outside a compact set.
We write $$u_i(x)=w_i(x)+\sum_{\ell=1}^m\beta_{i,\ell}\log|x-P_\ell|,\quad i=1,2. $$ Then $w_i$ solves \eqref{eq-w-toda} if and only if $u_i$ solves \begin{align}\label{eq-u-toda} \left\{ \begin{array}{ll}-\D u_1=2K_1e^{2u_1}-K_2e^{2u_2}\quad\text{in }\R^2 \\ -\D u_2=2K_2e^{2u_2}-K_1e^{2u_1}\quad\text{in }\R^2 \\ \rule{0cm}{.5cm}K_i(x):=\prod_{\ell=1}^m\frac{1}{|x-P_\ell|^{2\beta_{i,\ell}}}\quad i=1,2.\end{array}\right.\end{align}The condition  \eqref{behavior-w-toda} in terms of $u_i$ is  \begin{align}\label{asymp-u-toda}\left\{\begin{array}{ll}u_i(x)=-{\beta_i} \log|x|+\text{ a bounded continuous function} &\quad\text{on  }B_1^c\\ \rule{0cm}{.5cm} \beta_i:=2-\sum_{\ell=1}^m\beta_{i,\ell},&\quad i=1,2,\end{array}\right.\end{align} provided $u_i$ is continuous.

For Toda system with singular sources, the only complete result is \cite{LWY} in which the case of single source, i.e., $m=1$ is completely solved by PDE  and integrable system theory. In \cite{LNW}, some  special cases of $m=2$ are classified using higher order hypergeometric equations. The following theorem gives the {\em first} existence result when $m\geq 3$:

\begin{thm} \label{thm-toda}Let $m\geq 3$. Let $\{\beta_{i,\ell}:i=1,2,\, \ell=1,2,\dots m\}\subset[0,1)$ be such that \begin{align}\label{cond-beta-toda}3(1+\beta_{i,j})< 2\sum_{\ell=1}^m\beta_{i,\ell}+ \sum_{\ell=1}^m\beta_{3-i,\ell},\quad \sum_{\ell=1}^m\beta_{i,\ell}<2,\quad \text{for  }j=1,2,\dots,m,\,i=1,2.\end{align}  Then given $m$ distinct points $\{P_\ell\}_{\ell=1}^m\subset\R^2$ there exists continuous  solution $(u_1,u_2)$ to \eqref{eq-u-toda} such that  \eqref{asymp-u-toda} holds.   \end{thm}

Note that if $\sum_\ell{\beta_{1,\ell}}=\sum_\ell\beta_{2,\ell}$, then the first condition of \eqref{cond-beta-toda} reduces to $$\sum_{\ell=1}^m\beta_{1,\ell}>1+\beta_{i,j}\quad\text{for every }i=1,2,\, j=1,\dots,m,$$ which is stronger than \eqref{cond-beta}. We shall show that  an equivalent condition of  \eqref{cond-beta} for the Toda system, namely a condition of the form  \begin{align}\label{cond-beta-like}\sum_{\ell=1,\ell\neq j}^m\beta_{i,\ell}>\max\{\beta_{1,j},\beta_{2,j}\}\quad\text{for every }j=1,\dots,m,\,i=1,2,\end{align} is not sufficient  for the existence of solutions to   \eqref{eq-u-toda} satisfying the asymptotic behavior \eqref{asymp-u-toda}. See Lemma  \ref{lem-non-toda}.

\medskip

In \cite{Luo-Tian}, the existence of a solution to (\ref{eq-w}) is proved by a variational argument. In this paper we propose a new proof on the existence via fixed point theory. The crucial step in which we need  condition \eqref{cond-beta-toda} is  Proposition \ref{prop-toda} below, a compactness result which follows from the blow-up analysis of sequences of solutions (see Lemma \ref{lem-poho}). This compactness is used to prove the a priori bounds necessary to run the fixed point argument of \cite{Aviles,HMM,WY}. Let us point out that condition \eqref{cond-beta-like} is sufficient to rule-out  a ``full blow-up'' phenomena (that is, after a suitable rescaling, the limiting profile is a $SU(3)$ Toda system in $\R^2$) for a sequence of solutions to \eqref{eq-u-toda}-\eqref{asymp-u-toda} (for ``half blow-up'' and ``full blow-up'' phenomena see e.g.,  \cite{AW, DPR, MPW}). In particular, condition \eqref{cond-beta-like} is sufficient to prove the a priori estimate when  $\beta_{1,\ell}=\beta_{2,\ell}=\beta_\ell$ and $u_1=u_2$, that is, a priori estimate for the singular Liouville problem \eqref{eq-w-3} .   Moreover, the same method also works for a  higher order generalization of it. 

\begin{thm}\label{thm1} Let $m\geq 3$ and $n\geq 2$. For $\ell=1,2,\dots,m$ let $\beta_\ell\in (0,1)$   be such that  \eqref{cond-beta} holds. Then given $m$ distinct points $\{P_\ell\}_{\ell=1}^m\subset \R^n$ there exists a solution  $w\in C^0(\R^n\setminus\{P_1,\dots,P_m\})$    to $$(-\D)^\frac n2w=e^{nw}+\gamma_n\sum_{\ell=1}^m\beta_\ell \delta_{P_\ell}\quad\text{in }\R^n$$ satisfying the asymptotic behavior $$w(x)=-2\log|x|+O(1)\quad\text{as }|x|\to\infty.$$
\end{thm}

Here  $\gamma_n:=\frac{(n-1)!}{2}|S^n|$ is such that $$\frac{1}{\gamma_n}(-\D)^\frac{n}{2} \log\frac{1}{|x|}=\delta_0.$$



\section{Proof of Theorem \ref{thm-toda}}


It is well-know that if $(u_1,u_2)$ is a  solution to \eqref{eq-u-toda}    with $\beta_{i,\ell}<1$ and $u_i $, $K_ie^{2u_i}\in L^1_{loc}(\R^2)$, then $u_i$ is continuous.  On the other hand, if $(u_1,u_2)$ is a continuous solution to \eqref{eq-u-toda}-\eqref{asymp-u-toda} with $\beta_{i,\ell}<1$, then $K_ie^{2u_i}=O(|x|^{-4}) $ as $|x|\to\infty$. In particular, $\log|\cdot|K_ie^{2u_i}\in L^1(\R^2)$, and $u_i$ satisfies the integral equation  \begin{align}\label{eq-int-u}u_i(x):=\frac{1}{2\pi}\sum_{j=1}^2a_{i,j}\int_{\R^2}\log\left(\frac{1}{|x-y|}\right)K_j(y)e^{2(u_j(y))}dy +c_i,\quad i=1,2,\end{align}  for some $c_i\in\R$, where    $(a_{i,j})$ is  the $SU(3)$ Cartan matrix  $$ \begin{pmatrix} 2 & -1 \\ -1 & 2  \end{pmatrix}. $$ Moreover, the asymptotic behavior \eqref{asymp-u-toda} implies that $$\sum_{j=1}^2a_{i,j}\int_{\R^2}K_je^{2u_j}dx=2\pi\beta_i,\quad i=1,2,$$ that is  \begin{align}\label{int-cond-u}\int_{\R^2}K_ie^{2u_i}dx=2\pi\bar\beta_i,\quad \bar\beta_i:=\frac{1}{3}(2\beta_i+\beta_{3-i}),\quad i=1,2.\end{align}
Thus,   Theorem \ref{thm-toda} is equivalent to the existence of solution $(u_1,u_2)$ to \eqref{eq-int-u}-\eqref{int-cond-u}. Moreover,   \eqref{cond-beta-toda} in terms of $\bar\beta_i$ is \begin{align}\label{barbeta}\bar\beta_i>0,\quad  \bar\beta_i<1-\beta_{i,\ell}\quad\text{for every }i=1,2,\, \ell=1,\dots,m. \end{align}

In order to prove existence of solution to \eqref{eq-int-u}-\eqref{int-cond-u} we use a fixed point argument on the space  $$X: =C_0(\R^2)\times C_0(\R^2),\quad \|{\bf v}\|:=\max\{\|v_1\|_{L^\infty(\R^2)},\|v_2\|_{L^\infty(\R^2)}\}\quad\text{for }{\bf v}=(v_1,v_2)\in X,$$ where $C_0(\R^2)$ denotes the space of continuous functions vanishing at infinity.    We fix $u_0\in C^\infty (\R^2)$ such  that   $$u_0(x)=-\log|x|\quad\text{ on }B_1^c. $$ For $v\in C_0(\R^2)$ let  $c_{i,v}\in\R$ be  the unique number so  that   \begin{align}\label{eq-ci} \int_{\R^2}\bar{K_i}e^{2(v+c_{i,v})}dx=2\pi\bar\beta_i,\quad \bar K_i:=K_ie^{2\beta_i u_0},\quad i=1,2,\end{align} where $\bar\beta_i$ is as in \eqref{int-cond-u}.   Now we define $T:X\to  X$, $(v_1,v_2)\mapsto(\bar v_1,\bar v_2)$, where we have set \begin{align}\label{bar-vi} \bar v_i(x):=\frac{1}{2\pi}\sum_{j=1}^2a_{i,j}\int_{\R^2}\log\left(\frac{1}{|x-y|}\right)\bar K_j(y)e^{2(v_j(y)+c_{j,v_j})}dy-\beta_i u_0(x),\quad i=1,2.\end{align}
  As  $\beta_i=2\bar\beta_i-\bar\beta_{3-i}$, for $x\in B_1^c$ this can be written as $$ \bar v_i(x):=\frac{1}{2\pi}\sum_{j=1}^2a_{i,j}\int_{\R^2}\log\left(\frac{|x|}{|x-y|}\right)\bar K_j(y)e^{2(v_j(y)+c_{j,v_j})}dy,\quad i=1,2.$$ Using that $\bar K_i=O(|x|^{-4})$ for $|x|$ large,  one can show that $(\bar v_1,\bar v_2)\in X$.  Moreover, the operator $T$ is compact (see e.g. the proof of \cite[Lemma 4.1]{HMM}).

The following proposition is crucial in proving existence of fixed point of $T$.
\begin{prop} \label{prop-toda}There exists $C>0$ such that $$\|{\bf v}\|_{X} \leq C\quad\text{for every } ({\bf v},t)\in X\times[0,1]\text{ satisfying }{\bf v}=tT({\bf v}).$$ \end{prop}
\begin{proof}  We assume by contradiction that the proposition is false. Then there exists ${\bf v}^k=(v_1^k,v_2^k)$ and $t^k\in (0,1]$ with ${\bf v}^k=t^kT({ \bf v}^k)$  such that $\|{\bf v}^k\|\to\infty$.  We set $$\psi_i^k:=v_i^k+c_i^k,\quad c_i^k:=c_{i,v_i^k}+\frac12\log t^k.$$ Then we have \begin{align}\left\{\begin{array}{ll} \psi_1^k(x)=\frac{1}{2\pi} \int_{\R^2}\log\left(\frac{1}{|x-y|}\right)\left(2\bar K_1(y)e^{2\psi_1^k(y)}-\bar K_2(y)e^{2\psi_2^k(y)} \right)dy-t^k\beta_1 u_0(x)+c_1^k\\ \rule{0cm}{.8cm}\psi_2^k(x)=\frac{1}{2\pi} \int_{\R^2}\log\left(\frac{1}{|x-y|}\right)\left(2\bar K_2(y)e^{2\psi_2^k(y)}-\bar K_1(y)e^{2\psi_1^k(y)} \right)dy-t^k\beta_2 u_0(x)+c_2^k.\end{array} \right. \end{align} For $|x|\geq 1$ this is equivalent to \begin{align}\left\{\begin{array}{ll} \psi_1^k(x)=\frac{1}{2\pi} \int_{\R^2}\log\left(\frac{|x|}{|x-y|}\right)\left(2\bar K_1(y)e^{2\psi_1^k(y)}-\bar K_2(y)e^{2\psi_2^k(y)} \right)dy+c_1^k\\ \rule{0cm}{.8cm}\psi_2^k(x)=\frac{1}{2\pi} \int_{\R^2}\log\left(\frac{|x|}{|x-y|}\right)\left(2\bar K_2(y)e^{2\psi_2^k(y)}-\bar K_1(y)e^{2\psi_1^k(y)} \right)dy+c_2^k.\end{array} \right. \end{align}
Since $\|{\bf v}^k\|\to\infty$, we necessarily have $$\max\{\sup \psi_1^k,\sup \psi_2^k\}\to\infty. $$ Without any loss of generality we assume that $\sup \psi ^k_1\geq \sup \psi_2^k$.  We fix $x^k\in\R^2$ such that $$\sup\psi_1^k<\psi^k_1(x^{k})+1. $$  If $x^{k}$ is bounded then, up to a subsequence, $x^{k}\to x^\infty$.

 We consider the following three cases.

\medskip

\noindent\textbf{Case 1} $  x^\infty\in\R^2\setminus \{P_\ell:\ell=1,2,\dots,m\}$.

By Lemma \ref{lem-poho} (see also \cite{JLW, Lucia-Nolasco})  we have   $$\max\{\sigma_1(x^\infty),\sigma_2(x^\infty)\}\geq 1 ,$$ where the  blow-up value at a point $P$ is defined by  $$\sigma_i(P):=\lim_{r\to0}\lim_{k\to\infty}\frac{1}{2\pi}\int_{B_r(P)}\bar K_ie^{2\psi_i^k}dx,\quad i=1,2.$$   This   contradicts \eqref{barbeta} as $\sigma_i(x^\infty)\leq \bar\beta_i<1.$

\medskip

\noindent\textbf{Case 2} $  x^\infty\in \{P_\ell:\ell=1,2,\dots,m\}$.

Without loss of generality we  assume that $x^\infty=P_1$. Notice that $$\bar K_i(x)=\frac{f_i(x)}{|x-P_1|^{2\beta_{i,1}}},\quad i=1,2,$$ for some positive continuous functions  $f_1$ and $f_2$ in a small neighborhood of the point $P_1$. In particular,  the functions $w_i^k(x):=\psi_i^k(x-P_1)$ satisfies the conditions of Lemma \ref{lem-poho}  for some $R>0$, and  we get  $$\sigma_1(x^\infty)\geq    1- \beta_{11},\quad\text{or }\sigma_2(x^\infty)\geq 1-\beta_{2,1}  ,$$ a contradiction to \eqref{barbeta}.  

\medskip

\noindent\textbf{Case 3} $|x^{k}|\to\infty$.

We set $$\tilde \psi_i^k(x)=\psi_i^k(\frac{x}{|x|^2}),\quad \tilde K_i(x)=\frac{1}{|x|^4}\bar K_i(\frac{x}{|x|^2}) \quad\text{on }\R^2\setminus\{0\},\quad i=1,2,$$ and extend them continuously at the origin.  Then $\tilde \psi_i^k$ satisfies \begin{align}\left\{\begin{array}{ll}\tilde  \psi_1^k(x)=\frac{1}{2\pi} \int_{\R^2}\log\left(\frac{|y|}{|x-y|}\right)\left(2\tilde K_1(y)e^{2\tilde \psi_1^k(y)}-\tilde K_2(y)e^{2\tilde  \psi_2^k(y)} \right)dy+c_1^k\\ \rule{0cm}{.8cm}\tilde \psi_2^k(x)=\frac{1}{2\pi} \int_{\R^2}\log\left(\frac{|y|}{|x-y|}\right)\left(2\tilde  K_2(y)e^{2\tilde \psi_2^k(y)}-\tilde  K_1(y)e^{2\tilde \psi_1^k(y)} \right)dy+c_2^k,\end{array} \right. \end{align} for $x\in B_1$.     Since $\tilde K_i(0)>0$ for $i=1,2$,  and  $$\tilde\psi_1^k(\tilde x_k)\to\infty,\quad \tilde x_k:=\frac{x_k}{|x_k|^2}\to0,$$ one obtains a contradiction as in Case 1.

We conclude the proposition.

\end{proof}

\bigskip

\noindent\emph{Proof of Theorem \ref{thm-toda}} It follows from Proposition \ref{prop-toda} and Schauder fixed point theorem that the operator $T$ has a fixed point, say $(v_1,v_2)$.  Then setting $$u_i:=v_i+\beta_i u_0+c_{i,v_i},\quad i=1,2,$$ one sees that $(u_1,u_2)$ is a solution to \eqref{eq-u-toda}-\eqref{asymp-u-toda}.

\section{Non-existence results}

We show that Theorem \ref{thm-toda} is not true if the assumption \eqref{cond-beta-toda} is replaced by \eqref{cond-beta-like}.  Let us  fix $\beta_1,\dots,\beta_7\in (0,1)$ such that the assumptions  $\A1)$ to $\A5)$ hold:  \begin{itemize} \item[$\A1)$] $\beta_4+\sum_{\ell=1}^4\beta_\ell =2$\item[$\A2)$]  $\beta_2+\beta_3<\beta_1 $ \item[$\A3)$]  $\beta_4<\frac13$        \item[$\A4)$]  $ \beta_4+\sum_{\ell=5}^7 \beta_\ell=2$ \item[$\A5)$] $\beta_4+\beta_5<1$.\end{itemize}It is easy to see that $\A1)$ and $\A2)$ implies that \begin{itemize}\item[$\A6)$] $\beta_4+\beta_1>1$ and  $\beta_4+\beta_\ell< 1$  for $\ell=2,3.$ \end{itemize}

\medskip
 We shall show an non-existence result to the Toda system \eqref{eq-w-toda} satisfying \eqref{behavior-w-toda}    for the following choice of $\{\beta_{i,\ell}\}$:  \begin{align} \label{beta-14} \beta_{1,\ell}:=\left\{\begin{array}{ll} \beta_\ell\quad\text{for }\ell=1,2,3,4\\ 0\quad\text{for }\ell=5,6,7\end{array}\right.,\quad   \beta_{2,\ell}:=\left\{\begin{array}{ll}  0\quad\text{for }\ell=1,2,3,4\\ \beta_\ell\quad\text{for }\ell=5,6,7.\end{array}\right. \end{align}

Let us   point out that   we can choose $\{\beta_{\ell}\}$  satisfying $\A1)$ to $\A5)$ in such a way that $\{\beta_{i,\ell}\}$ satisfy \eqref{cond-beta-like} with $m=7$, $i=1,2$. For instance, one can simply take $$\beta_1=1-\ve,\,\beta_2=\beta_3=\frac12-\ve,\,\beta_4=\frac{3\ve}{2},\,\beta_5=1-\frac{5\ve}{2},\,\beta_6=\beta_7=\frac{1+\ve}{2},\,\ve\in(0,\frac29).$$   For these $\beta_\ell$'s one has $$\sum_{\ell=1}^7\beta_{1,\ell}=(1+\beta_{1,1})-\frac\ve2,$$ and hence $\{\beta_{i,\ell}\}$ does not satisfy \eqref{cond-beta-toda}.

We begin with the following non-existence result for a singular Liouville equation.

\begin{lem}\label{lem-non} Let $\beta_\ell\in (0,1)$ with $\ell=1,2,3,4$ be such that  $\A1)$ to $\A3)$ hold.  Let $P_1,P_2,P_3$ be fixed three distinct points in $\R^2$. Then, for $|P_4|$ large enough,  there exists no continuous solution to \begin{align}\label{eq-non}-\D u=\prod_{\ell=1}^4\frac{1}{|x-P_\ell|^{2\beta_\ell}}e^{2u}\quad\text{in }\R^2,\quad u(x)=-2\beta_4\log|x|+O(1)\quad\text{as }|x|\to\infty.\end{align}\end{lem}\begin{proof} Assume by contradiction that there exists a sequence of solutions $(u^k)$  to \eqref{eq-non} with $$P_4=P_{4,k},\quad |P_4|\to\infty \quad \text{as }k\to\infty.$$ Notice that the assymptotic behavior  $$u^k(x)=-2\beta_4\log|x|+O_k(1)\quad\text{as }|x|\to\infty$$   is equivalent to  \begin{align*}\int_{\R^2}\frac{K_0(x)}{|x-P_4|^{2\beta_4}}e^{2u^k}dx=4\pi\beta_4,\quad K_0(x):=\prod_{\ell=1}^3\frac{1}{|x-P_\ell|^{2\beta_\ell}}.\end{align*}

\medskip
\noindent\textbf{Step 1} We have  $$\lim_{R\to\infty}\lim_{k\to\infty}\int_{B_R^c}\frac{K_0(x)}{|x-P_4|^{2\beta_4}}e^{2u^k(x)}dx=0.$$ To prove this we use Kelvin transform. Up to a small translation, we can  assume that none of  $P_1,P_2,P_3$  is the origin.   We set $$\tilde u^k(x):=u^k(\frac{x}{|x|^2})-2\beta_4\log|x|+c^k,\quad x\neq0,$$ for some $c^k\in\R$. Then setting    $Q_\ell:=\frac{P_\ell}{|P_\ell|^2}$  for $\ell=1,2,3,4$ we see that  $$-\D \tilde u^k(x)=\frac{1}{|x|^4}\prod_{\ell=1}^4\frac{1}{|\frac{x}{|x|^2}-\frac{Q_\ell}{|Q_\ell|^2}|^{2\beta_\ell}}e^{2u^k(\frac{x}{|x|^2})}\quad\text{in }\R^2\setminus\{0\}.$$ Using that $|x||y||\frac{x}{|x|^2}-\frac{y}{|y|^2}|=|x-y|$, $\A1)$,  and for suitably chosen $c^k$, we obtain \begin{align*}-\D \tilde u^k(x)=|x|^{2\beta_4}\prod_{\ell=1}^4\frac{1}{|x-Q_\ell|^{2\beta_\ell}}e^{2\tilde u^k(x)}\quad\text{in }\R^2\setminus\{0\} \\ \tilde u^k(x)=-2\beta_4\log|x|+O_k(1)\quad\text{as }|x|\to\infty.\end{align*}In fact, as $\tilde u^k=O_k(1)$  in $B_1$, it satisfies the above equation at the origin as well, that is, \begin{align*}-\D \tilde u^k(x)=\frac{|x|^{2\beta_4}}{|x-Q_4|^{2\beta_4}}f(x)e^{2\tilde u^k(x)}\quad\text{in }\R^2,\quad f(x):=\prod_{\ell=1}^3\frac{1}{|x-Q_\ell|^{2\beta_\ell}}.\end{align*} As $|P_4|\to\infty$, we have that  $Q_4\to0$. By $\A3)$ one gets   \begin{align}\label{non-vol}\int_{\R^2}\frac{|x|^{2\beta_4}}{|x-Q_4|^{2\beta_4}}f(x)e^{2\tilde u^k(x)}=4\pi\beta_4\leq 2\pi(1-\beta_4-\ve)\end{align} for some $\ve>0$.  Hence, by Lemma \ref{bound-upper} we obtain $$\tilde u^k\leq C\quad\text{in }B_\delta\quad\text{for some }\delta>0.$$ Step 1 follows immediately from the relation $$\int_{B_R^c}\frac{K_0(x)}{|x-P_4|^{2\beta_4}}e^{2u^k(x)}dx=\int_{B_\frac1R}\frac{|x|^{2\beta_4}}{|x-Q_4|^{2\beta_4}}f^k(x)e^{2\tilde u^k(x)}dx.$$

\medskip
\noindent\textbf{Step 2} No blow-up occurs on bounded domains, that is, for every $R>0$, $$u^k-\beta_4\log |P_4|\leq C(R)\quad\text{on }B_R.$$

Writing $\bar u^k=u^k-\beta_4\log|P_4|$ we see that $$-\D \bar u^k=K_0K_1e^{2\bar u^k}\quad\text{in }\R^2,\quad \int_{\R^2}K_0K_1e^{2\bar u^k}dx=4\pi\beta_4,$$ where $$K_0(x):=\prod_{\ell=1}^3\frac{1}{|x-P_\ell|^{2\beta_\ell}},\quad K_1:=\frac{|P_4|^{2\beta_4}}{|x-P_4|^{2\beta_4}}.$$ It follows that $K_1\to1$ in $C^0_{loc}(\R^2)$ as $k\to\infty$, and $K_0$ does not depend on $k$.

Assume by contradiction that $\bar u^k$ is not locally uniformly bounded from above. Then, as blow-up points are discrete, there exists $\delta>0$ such that $$\max_{B_\delta(x_0)}\bar u^k=\bar u^k(x^k)\to\infty,\quad x^k\to x_0,$$ for some $x_0\in \R^2$. If $x_0\not\in\{P_2,P_3,P_4\}$, then one can show that $$4\pi\beta_4\geq\lim_{r\to0}\lim_{k\to\infty}\int_{B_r(x_0)}K_0K_1e^{2\bar u^k}dx\geq 4\pi,$$ a contradiction as $\beta_4<1$. Thus, $x_0=P_{\ell_0}$ for some $\ell_0\in\{1,2,3\}$, and in fact, the set of all blow-up points is a subset of $\{P_1,P_2,P_3 \}$.  We fix $R>0$ such that $\bar B_{2R}(x_0)\cap\{P_1,P_2,P_3 \}=\{x_0\}$. Then $\bar u^k$ is uniformly bounded from above  in $B_{2R}(x_0)\setminus B_\frac R2(x_0)$.  Using this, and as
 $\bar u^k $ satisfies the integral equation $$\bar u^k(x)=\frac{1}{2\pi}\int_{\R^2}\log\left(\frac{1+|y|}{|x-y|}\right)K(y)e^{2\bar u^k(y)}dy+C^k,\quad K:=K_0K_1,$$ for some $C^k\in\R$, we get that  $$|\bar u^k(x)-\bar u^k(y)|\leq C\quad\text{for every }x,y\in \partial B_R(x_0).$$ Hence, by the remark after  Lemma \ref{lem-poho} we have (this can  be shown easily by a local Pohozaev type identity  to the above integral equation satisfied by $\bar u_k$) $$\sigma(x_0)=\lim_{r\to0}\lim_{k\to\infty}\frac{1}{2\pi}\int_{B_r(x_0)}K_0K_1e^{2\bar u^k}dx=2(1-\beta_{\ell_0}).$$ Thus  $2\beta_4\geq \sigma(x_0)=2(1-\beta_{\ell_0})$. This  and $\A6)$ imply  that $\ell_0=1$, that is, $P_1$ is the only blow-up point.    In particular, $\bar u^k\to-\infty$  locally uniformly outside $P_1$. Therefore, by Step 1 and \eqref{non-vol}     we get $$2\beta_4=\frac{1}{2\pi}\lim_{k\to\infty}\int_{\R^2}K_0K_1e^{2\bar u^k}dx=\sigma(x_0)=2(1-\beta_1),$$ a contradiction to $\A6)$.   This finishes Step 2.

\medskip

Since $\bar u^k$ is locally uniformly bounded from above, up to a subsequence, either $\bar{u}^k \to \infty$ locally uniformly, or $\bar u^k\to\bar u$ in $C_{loc}^0(\R^2)$. In the first case we get a  contradiction to $$\int_{\R^2}K_0K_1e^{2\bar u^k}dx=4\pi\beta_4,\quad \lim_{R\to\infty}\lim_{k\to\infty}\int_{B_R^c}K_0K_1e^{2\bar u^k}dx=0,$$ thanks to Step 1. Therefore, only the later case can occur, and the limit function $\bar u$ satisfies $$-\D \bar u=K_0e^{2\bar u}\quad\text{in }\R^2,\quad K_0=\prod_{\ell=1}^3\frac{1}{|x-P_\ell|^{2\beta_\ell}}.$$ Again by Step 1, we have that $$\int_{\R^2}K_0e^{2\bar u}dx=4\pi\beta_4,$$ which is equivalent to $$\bar u(x)=-2\beta_4\log|x|+O(1)\quad\text{as }|x|\to\infty.$$ Thus, $$w(x):=\bar u(x)-\sum_{\ell=1}^3\beta_\ell \log|x-P_\ell|$$ satisfies \eqref{eq-w-3} with $m=3$, where $\beta_1,\beta_2,\beta_3$ satisfy $\A2)$. This contradicts the necessary condition \eqref{cond-beta} in Theorem \ref{Luo-Tian}.
\end{proof}

\begin{rem}Problem \eqref{eq-non} is super critical under the assumptions $\A1)$ and $\A2)$. To be more precise, if one uses fixed point arguments (as described in Section \ref{sec-Lio}) to prove the lemma, then one would not be able to rule-out a  blow-up phenomena around  the point $P_1$. This is due to the fact that the energy of a singular bubble at $P_1$ is $4\pi(1-\beta_1)$, which is smaller than   the total energy $4\pi\beta_4$.

The super criticality of the  Problem \eqref{eq-non} under  $\A1)$ and $\A2)$  can also be seen from the point of view of singular Moser-Trudinegr inequality, see e.g. \cite{Adi,BM,Malchiodi-Ruiz,WXChen,Tr} and the references therein.   \end{rem}

Now we are in a position to prove non-existence of solution to the Toda system  \eqref{eq-w-toda}-\eqref{behavior-w-toda} for the choice of $\{\beta_{i,\ell}\}$ as in \eqref{beta-14}. More precisely, we have:

\begin{lem} \label{lem-non-toda}Let $\beta_\ell\in (0,1)$ with $\ell=1,\dots, 7$ be such that  $\A1)$ to $\A5)$ hold. Let $\{\beta_{i,\ell}:i=1,2,\,\ell=1,\dots, 7\}$ be as in \eqref{beta-14}. Let  $P_1,\dots,P_4$ be such that Problem \eqref{eq-non} has no solution. Let $P_5$ be a fixed point (different from $P_1,\dots, P_4$).  Then for $|P_6|, |P_7|$ large ($P_6\neq P_7$) there exists no solution to    \eqref{eq-u-toda} with $m=7$   such that $$u_i(x)=-\beta_4\log|x|+O(1)\quad\text{as }|x|\to\infty,\quad i=1,2.$$ \end{lem}
\begin{proof} We assume by contradiction that   there is a sequence of solutions  $(u_i^k)$ with $$P_\ell=P_{\ell,k},\quad |P_\ell|\xrightarrow{k\to\infty}\infty\quad\text{for }\ell=6,7,$$  that is, $u_i^k$   satisfies \begin{align}\left\{\begin{array}{ll} -\D u_1^k=2K_1e^{2u_1^k}-K_2e^{2u^k_2} &\quad\text{in }\R^2\\-\D u^k_2=2K_2e^{2u^k_2}-K_1e^{2u^k_1}&\quad\text{in }\R^2\\ \rule{0cm}{.6cm}\int_{\R^2}K_ie^{2u^k_i}dx=2\pi\beta_4 &\quad i=1,2 \\ \rule{0cm}{.5cm} |P_\ell|\xrightarrow{k\to\infty}\infty &\quad \ell=6,7,\end{array}\right.\end{align} where $$K_1(x):=\prod_{\ell=1}^4\frac{1}{|x-P_\ell|^{2\beta_\ell}},\quad K_2(x):=\prod_{\ell=5}^7\frac{|P_6|^{2\beta_6}|P_7|^{2\beta_7}}{|x-P_\ell|^{2\beta_\ell}}. $$ Notice that $K_1$ does not depend on $k$, $K_1\in L^1(\R^2)$, thanks to the assumption $\beta_4<1$, and $$K_2\to |x-P_5|^{-2\beta_5}\text{ locally uniformly in }\R^2\setminus\{P_5\}\quad\text{ as }k\to\infty.$$

  We claim that $u_1^k\to u$ locally uniformly in $\R^2$, where $u$ satisfies \begin{align}\label{lim-eq-single}-\D u=2K_1e^{2u}\quad\text{in }\R^2,\quad \int_{\R^2}K_1e^{2u}dx=2\pi\beta_4.\end{align} Then one can show that $u(x)=-2\beta_4\log|x|+O(1)$ as $|x|\to\infty$. In particular, $\bar u(x)=u(x)+\frac12\log2$ is a solution to the Problem \eqref{eq-non}, a contradiction to  our assumption on $P_1,\dots,P_4$ that  the Problem \eqref{eq-non} has no solution.

We prove the claim in few steps.

\medskip

\noindent \textbf{Step 1} We have $$\lim_{R\to\infty}\lim_{k\to\infty}\int_{B_R^c}K_1e^{2u_1}dx=0.$$

The proof is very similar to that of Step 1 in Lemma \ref{lem-non}. Here we give a sketch of it.

 We set $$\tilde u_1^k(x)=u^k_1(\frac{x}{|x|^2})-\beta_4\log|x|+c^k,$$ so that $\tilde u^k_1$ satisfies ($\tilde K$ does not depend on $k$) $$-\D \tilde u^k_1=\tilde Ke^{2\tilde u^k_1}-g^k\quad\text{in }\R^2, \quad \int_{\R^2}\tilde Ke^{2\tilde u^k}dx=4\pi\beta_4, \quad\int_{\R^2}g^kdx=2\pi\beta_4,  $$ $$g^k,\tilde K>0\quad\text{in }\R^2,\quad \tilde K(x)\xrightarrow{|x|\to0}1.$$ Now we can apply Lemma \ref{bound-upper}  with $\beta=0$,  thanks to the assumption $\A3)$, to   get that $\tilde u^k_1\leq C$ in a neighborhood of the origin.   Step 1 follows.

\medskip

Setting $$S_i:=\{x\in\R^2:\text{ there is a sequence }x^k\to x\text{ such that }u_i^k(x^k)\to\infty\},\quad i=1,2,$$ we   shall show that $S_1\cup S_2=\emptyset$. We start with:

\noindent \textbf{Step 2} $S_1\subseteq\{P_1,\dots,P_4\}$ and  $S_2\subseteq\{P_5\}$.

For $x_0\in S_1\cup S_2$ we can write    $$ K_i(x) =\frac{c_i+o(1)}{|x-x_0|^{2\alpha_i}}, \quad c_i>0,\quad o(1)\xrightarrow{x\to x_0}0,\quad i=1,2,$$ where $\alpha_1\in\{0,\beta_1,\dots,\beta_4\}$, $\alpha_2\in\{0,\beta_5\}$ and $\alpha_1\alpha_2=0$.  By Lemma \ref{bound-upper} and $\A3)$ one gets $S_1\subseteq\{P_1,\dots,P_4\}$ and $S_2\subseteq\{P_5\}$.

\medskip

\noindent \textbf{Step 3} $S_1\cup S_2=\emptyset$.

It is well-known that $u^k_i$ satisfies the integral equation $$u^k_i(x)=\frac{1}{2\pi}\int_{\R^2}\log\left(\frac{1+|y|}{|x-y|}\right)\left(2K_i(y)e^{2u^k_i(y)}-K_{3-i}(y)e^{2u^k_{3-i}(y)}\right)dy+C^k,\quad i=1,2.$$ For $x_0\in S_1\cup S_2$ let $R>0$ be such that $\bar B_R(x_0)\cap (S_1\cup S_2)=\{x_0\}$, and $x_0$ is the only singularity for $K_1,K_2$ on $\bar B_R(x_0)$. Then, from the above integral representation, one can show that $$|u^k_i(x)-u^k_i(y)|\leq C\quad \text{for every }x,y\in\partial B_R(x_0),\quad i=1,2.$$ In particular, $u^k_i$ and $K_i$  satisfy all the assumptions in Lemma \ref{lem-poho}. Therefore, if  $S_2=\{P_5\}$, then as $\sigma_1(P_5)=0$, we must have  $\sigma_2(P_5) =1-\beta_5$. This implies that $$\beta_4\geq \sigma_2(P_5)=1-\beta_5,$$  a contradiction to $\A5)$. Hence, $S_2=\emptyset$.

Now we assume that $\beta_{\ell_0}\in S_1$ for some $\ell_0\in \{1,\dots,4\}$. Then, in a similar way we get  that $\beta_4 \geq 1-\beta_{\ell_0}$. In fact, by $\A6)$, a strict inequality holds, that is, $\beta_4> 1-\beta_{\ell_0}$.  Since $$u_1^k\to-\infty \quad\text{locally uniformly in }\R^2\setminus S_1,  $$ we must have that the cardinality of $S_1$ is at least $2$, thanks to   Step 1. Taking $P_{\ell_1}\in S_1$ with $\ell_1\in\{1,\dots,4\}\setminus\{\ell_0\}$, and again using that $\sigma({P_{\ell_1}})= 1-\beta_{\ell_1}$, we obtain $$\beta_4\geq \sigma(P_{\ell_0})+\sigma(P_{\ell_1})=2-\beta_{\ell_0}-\beta_{\ell_1},$$ a contradiction to $\A1)$.

 We conclude Step 3.

\medskip
\noindent \textbf{Step 4} $u_1^k\to\bar u_1$ in $C^0_{loc}(\R^2)$ where $\bar u_1$ satisfies  \eqref{lim-eq-single}.

Since $S_1\cup S_2=\emptyset$, up to a subsequence,  one of the following holds: \begin{itemize}\item[i)]  $u_i^k\to\bar u_i$ in $C^0_{loc}(\R^2)$ for $i=1,2$  \item[ii)]  $u_1^k\to\bar u_1$ in $C^0_{loc}(\R^2)$  and $u_2^k\to-\infty$ locally uniformly in $\R^2$  \item[iii)] $u_2^k\to\bar u_2$ in $C^0_{loc}(\R^2)$  and $u_1^k\to-\infty$ locally uniformly in $\R^2$ \item[iv)] $u_i^k\to-\infty$ locally uniformly in $\R^2$ for $i=1,2$.\end{itemize} It follows from Step 1, and the integral condition $\int_{\R^2}K_1e^{2u_1}dx=2\pi\beta_4$ that  either $i)$ or $ii)$ holds, and $\bar u_1$ satisfies the integral condition $$\int_{\R^2}K_1e^{2\bar u_1}dx=2\pi\beta_4.$$

Now we assume by contradiction that $i)$ holds. Then the limit functions $(\bar u_1,\bar u_2)$ satisfy the system  \begin{align}\label{lim-eq}\left\{\begin{array}{ll} -\D \bar u_1=2 K_1 e^{2\bar u_1 }-\bar K_2 e^{2\bar u_2 }&\quad\text{in }\R^2\\ \rule{0cm}{.7cm}-\D \bar u_2 =2\bar K_2e^{2\bar u_2 }- K_1 e^{2\bar u_1 } &\quad\text{in }\R^2\\ \rule{0cm}{.6cm}\int_{\R^2} K_1e^{2\bar u_1}dx=2\pi\beta_4,\quad  \int_{\R^2} \bar K_2e^{2\bar u_2}dx=:2\pi\gamma\leq2\pi\beta_4,  \end{array}\right.\end{align} where $\bar K_2(x):=|x-P_5|^{-2\beta_5}$ is the limit of $K_2$ as $k\to\infty$. Then one has $$\lim_{|x|\to\infty}\frac{\bar u_2(x)}{\log|x|}=-(2\gamma-\beta_4),$$ and together with $\bar K_2e^{2\bar u_2}\in L^1(\R^2)$ we have $\beta_5+2\gamma-\beta_4>1$. Hence, $\beta_4+\beta_5>1$, a contradiction to $\A5)$.

Thus, $ii)$ holds, and \eqref{lim-eq} reduces to a single equation \eqref{lim-eq-single}.

 We conclude the lemma.
\end{proof}

\section{Higher order singular Liouville equation}\label{sec-Lio}

The proof of  Theorem \ref{thm1} is very similar to that of Theorem \ref{thm-toda} (see also \cite{HMM}). Here we give a sketch of it.

Writing $$w(x)=u(x)-\sum_{\ell=1}^m\beta_\ell\log|x-P_\ell|,$$ Theorem \ref{thm1} is equivalent to prove the existence of solution $u\in C^0(\R^n)$ to
\begin{align}\label{eq-u}  (-\D)^\frac{n}{2} u=Ke^{nu}\quad\text{in }\R^n,\quad K(x):=\prod_{\ell=1}^m\frac{1}{|x-P_\ell|^{n\beta_\ell}},\end{align} satisfying the asymptotic behavior \begin{align}\label{asymp-u}u(x)=-\beta\log|x|+O(1)\quad\text{as }|x|\to\infty,\quad \beta:=2-\sum_{\ell=1}^m\beta_\ell.\end{align}

As before we fix  $u_0\in C^\infty(\R^n)$ such that $u_0(x)=-\log|x|$ for $ |x|\geq 1$, and  we look for a solution $u$ to \eqref{eq-u} of the form $$u=\beta u_0+v+c, $$ where $c$ is a normalizing constant and $v\in X$, where  $$X:=C_0({\R^n})=\{v\in C^0(\R^n):v(x)\xrightarrow{|x|\to\infty}0 \},\quad \|v\|:=\max_{x\in \R^n}|v(x)|.$$  Then $u$ satisfies \eqref{eq-u} if and only if    $v=u-\beta u_0-c$ satisfies \begin{align} (-\D)^\frac{n}{2}  v=\bar K e^{nv+c}-\beta(-\D)^\frac n2 u_0\quad\text{in }\R^n,\quad \bar K:=Ke^{n\beta u_0}.\end{align}  The function   $\bar K$ satisfies  \begin{align} \label{asymp-barK} \lim_{|x|\to\infty}|x|^{2n}\bar K(x)=1.\end{align} For $v\in X$,  we fix     $c_v\in\R$ so that \begin{align}\label{def-cv} \int_{\R^n}\bar K(x)e^{n(v(x)+c_v)}=\beta \gamma_n.\end{align} We define a compact operator  $$T:X\to X,\quad v\mapsto\bar v,$$ \begin{align}\label{def-barv}\bar v(x):=\frac{1}{\gamma_n}\int_{\R^n}\log\left(\frac{1}{|x-y|}\right)\bar K(y)e^{n(v(y)+c_v)}dy-\beta u_0(x),\quad x\in\R^n. \end{align} It follows that $\bar v\in C^0(\R^n)$ (in fact, H\"older continuous), and by \eqref{def-cv}$$\bar v(x)=\frac{1}{\gamma_n}\int_{\R^n}\log\left(\frac{|x|}{|x-y|}\right)\bar K(y)e^{n(v(y)+c_v)}dy\quad\text{for }|x|>1.$$

We claim that   there exists $C>0$ such that \begin{align}\label{claim1}\|v\|_X\leq C\quad \text{for every }(v,t)\in X\times [0,1]\quad\text{satisfying }v=tT(v).  \end{align} Then by Schauder fixed point theorem the operator $T$ has a fixed point $v$ in $X$, and consequently we get  a continuous solution to \eqref{eq-u} satisfying \eqref{asymp-u}.

To prove \eqref{claim1} we assume by contradiction that there exists $(v^{k},t^{k})\in X\times [0,1]$ such that  $\|v^{k}\|_X\to\infty$ and  $v^{k}=t^{k} T(v^{k})$, that is \begin{align}\label{vk}  v^{k}(x)=\frac{t^{k}}{\gamma_n}\int_{\R^n}\log\left(\frac{1}{|x-y|}\right)\bar K(y)e^{n(v^{k}(y)+c_{v^{k}})}dy-t^{k}\beta u_0(x).\end{align}  Then we can choose  $x^k\in\R^n$ so that    $$\sup_{x\in\R^n}\psi^{k}(x)\leq \psi^{k}(x^{k})+1\xrightarrow{k\to\infty}\infty,\quad  \psi^{k}(x):=v^{k}(x)+c_{v^{k}}+\frac1n\log t^{k}.$$   The  crucial ingredients to obtain a contradiction are  Lemma \ref{thmA},  and  the relation  \begin{align}\label{est-beta} \beta=2-\sum_{\ell=1}^m\beta_\ell=2-\beta_\ell-\sum_{\ell\neq j}\beta_\ell<2(1-\beta_j)\quad\text{for every }j=1,2,\dots,m,\end{align} which follows from  the second condition in \eqref{cond-beta}.   Up to a subsequence,    we distinguish  the  following two cases:

\medskip
\noindent\textbf{Case 1} $x^k\to x^\infty\in \R^n\setminus\{P_\ell:\ell=1,2,\dots.m\}.$

In a small neighborhood of $x^\infty$ we have for some $c_0>0$ $$\bar K(x)=\frac{c_0+o(1)}{|x-x^\infty|^{n\alpha}},\quad o(1)\xrightarrow{x\to x^\infty}0, $$ where $\alpha\in \{0,\beta_1,\dots,\beta_m\}$.  Using \eqref{vk}-\eqref{est-beta} one gets a contradiction as in \cite{HMM}, see also \cite{Aviles, WY}.

\medskip
\noindent\textbf{Case 2} $|x^{k}|\to \infty$.

Setting  $$\tilde \psi^{k}(x):=\psi^{k}(\frac{x}{|x|^2}),\quad \tilde x^k:=\frac{x^k}{|x^k|^2}\to0,$$ we obtain  $ \tilde \psi_k(\tilde x^k)\to\infty$, and $\tilde \psi_k$ satisfies  \begin{align*} \tilde \psi^{k}(x)=\frac{1}{\gamma_n}\int_{\R^n}\log\left(\frac{|y|}{|x-y|}\right)\tilde K(y)e^{n\tilde \psi^{k}(y)}dy+c^k\quad \text{in }B_1,\end{align*} where $$\tilde K(x):=\frac{1}{|x|^{2n}}\bar K(\frac{x}{|x|^2}),\quad c^k:=c_{v^{k}}+\frac1n\log t^{k}.$$ Note that $\tilde K$ is smooth around the origin and $$\tilde K(x)\xrightarrow{|x|\to0}1, $$ one can proceed as in Case 1.
Thus, $\psi^{k}\leq C$ on $\R^n$, and we have \eqref{claim1}.

\section{Some useful lemmas}
The following lemma is a generalizations of Brezis-Merle \cite{Brezis-Merle} type results, compare \cite[Theorem 5]{Bar-Tar2}.
\begin{lem} \label{bound-upper}Let $(u^k)$ be a sequence of solutions to $$-\D u^k=\frac{f^k(x)}{|x|^{2\alpha}}e^{2u^k}-g^k\quad \text{in }B_1,\quad  \int_{B_1}\frac{f^k(x)}{|x |^{2\alpha}}e^{2u^k}dx\leq 2\pi(1-\alpha-\ve),$$ for some $\ve>0$ and  $\alpha\in [0,1)$. Assume that   $g^k\geq0$,  $\|g^k\|_{L^1(B_1)}\leq C$, $0\leq f^k\leq C$ and   $\inf _{B_1\setminus B_\delta}f^k\geq C_\delta^{-1}$ for some $0<\delta<\frac13$.  Then $u^k$ is locally uniformly bounded from above in $B_1$. \end{lem}
\begin{proof} We write $u^k=v^k+h^k$, where $h^k$ is harmonic in $B_1$ and  $$v^k(x):=\frac{1}{2\pi}\int_{B_1}\log\left(\frac{2}{|x-y|}\right)\left(\frac{f^k(y)}{|y|^{2\alpha}}e^{2u^k(y)}-g^k(y)\right)dy.$$   Since $g_k\geq 0$, by Jensen's inequality one gets that $$\int_{B_1}e^{2pv^k(x)}dx\leq C(p),\quad p\in [1,\frac{1}{1-\alpha-\ve/2}].$$ Notice that  $$\int_{B_1\setminus B_\delta}(h^k)^+dx\leq \int_{B_1\setminus B_\delta}((u^k)^++|v^k|)dx\leq C.$$  Since $\delta<\frac13$, fixing  $\delta+\frac13<r_1<r_2<1-\delta$ we see   that $$\partial B_t(x)\subset B_1\setminus B_\delta\quad \text{for every }x\in\bar B_\delta,\,r_1\leq t\leq r_2.$$Therefore, by mean value theorem, $$2\pi(r_2-r_1)h^k(x)=\int_{r_1}^{r_2}\int_{\partial B_t(x)}h^k(y)d\sigma(y)dt\leq \int_{B_1\setminus B_\delta}(h^k)^+dy\leq C.$$ Thus, $\int_{B_1}(h^k)^+dx\leq C$. If  $\rho^k:=\int_{B_\frac12}|h^k|dx\leq C$ then we have  $$h^k\to h\quad\text{in }C^2_{loc}(B_1),\quad \D h=0\quad\text{in }B_1.$$ In particular, $(h^k)$ is bounded in $C^0_{loc}(B_1)$. If $\rho^k\to\infty$, then $$\frac{h^k}{\rho^k}\to h\quad\text{in }C^2_{loc}(B_1),\quad \D h=0,\quad h<0\quad\text{in }B_1.$$ This shows that $(h^k)$ is locally uniformly bounded from above in $B_1$.  This leads to $$\int_{B_r}e^{2pu^k}dx\leq C_r\int_{B_r}e^{2pv^k}dx\leq C(p,r,\ve,\alpha),\quad 0<r<1,\, p\in [1,\frac{1}{1-\alpha-\ve/2}].$$  Using this uniform bound, and H\"older inequality with $p=\frac{1}{1-\alpha-\ve/2}$, one gets $v^k\leq C$ in $B_r$ for $0<r<1$, and the lemma follows.   \end{proof}

A strong version (precise quantization value of $\sigma_1,\sigma_2$) of the following lemma is proven in \cite{LWYZ,LWZ}. See  \cite{Lucia-Nolasco} for  a Pohozaev type identity for regular $SU(3)$ Toda system.
\begin{lem}[\cite{LWYZ,LWZ}] \label{lem-poho}Let $(u_1^k,u_2^k)$ be a sequence of solutions to  \begin{align}\left\{\begin{array}{ll} -\D u_1^k=2\frac{K_1^k}{|x|^{2\alpha_1}}e^{2u_1^k}-\frac{K_2^k}{|x|^{2\alpha_2}}e^{2u_2^k}&\quad\text{in }B_1\\ \rule{0cm}{.7cm}-\D u_2^k=2\frac{K_2^k}{|x|^{2\alpha_2}}e^{2u_2^k}-\frac{K_1^k}{|x|^{2\alpha_1}}e^{2u_1^k} &\quad\text{in }B_1\\ \rule{0cm}{.6cm}\int_{B_1}\frac{K_i^k}{|x|^{2\alpha_i}}e^{2u_i^k}dx\leq C&\quad  i=1,2 \\ \rule{0cm}{.5cm} |u_i^k(x)-u_i^k(y)|\leq C &\quad \text{for every }x,y\in\partial B_1,\quad i=1,2 \\ \rule{0cm}{.5cm} \|K_i^k\|_{C^3(B_1)}\leq C, \quad 0<\frac1C\leq K_i^k &\quad\text{in }B_1,\quad i=1,2,\end{array}\right.\end{align}  for some $\alpha_1,\alpha_2<1$, and $B_1$ is the unit ball in $\R^2$. Assume that $0$ is  the only blow-up point, that is, $$\sup_{B_1\setminus B_\ve}u_i^k\leq C(\ve)\quad\text{for every }0<\ve<1,\quad i=1,2. $$ Then setting $$\sigma_i:=\lim_{r\to0}\lim_{k\to\infty}\frac{1}{2\pi}\int_{B_r}\frac{K_i^k(x)}{|x|^{2\alpha_i}}e^{2u_i^k(x)}dx,\quad i=1,2,$$ we have $$\sigma_1^2+\sigma_2^2-\sigma_1\sigma_2=\sigma_1(1-\alpha_1)+\sigma_2(1-\alpha_2).$$  In particular,   if $(\sigma_1,\sigma_2)\neq (0,0)$ then $$\sigma_1\geq 1-\alpha_1\quad\text{or }\sigma_2\geq 1-\alpha_2.$$  \end{lem}

\begin{rem}  If $\alpha_1=\alpha_2=\alpha$, $K_1^k=K_2^k$ and  $u_1^k=u_2^k$ in the above lemma,  then  $\sigma_1=\sigma_2=2(1-\alpha)$. \end{rem}
\begin{thm}[\cite{HMM,PT}] \label{thmA}Let $u$ be a normal solution to \begin{align}\label{eq-classi}(-\D)^\frac n2u=|x|^{n\alpha}e^{nu}\quad\text{in }\R^n,\quad \Lambda:=\int_{\R^n}|x|^{n\alpha}e^{nu}dx<\infty,\end{align} for some $\alpha>-1$ and $n\geq 2$, that is, $u$ satisfies the integral equation  $$u(x)=\frac{1}{\gamma_n}\int_{\R^n}\log\left(\frac{1+|y|}{|x-y|}\right)|y|^{n\alpha}e^{nu(y)}dy+C,$$ for some $C\in\R.$   Then $\Lambda=\Lambda_1(1+\alpha)$, $\Lambda_1:=2\gamma_n$.  \end{thm}



\end{document}